\documentclass[12pt,reqno]{amsart}
\usepackage{mathrsfs}
\usepackage{amsgen}
\usepackage{amscd}
\usepackage{amsmath}
\usepackage{latexsym}
\usepackage{amsfonts}
\usepackage{amssymb}
\usepackage{amsthm}
\usepackage{graphicx}
\usepackage{color}
\usepackage{amsthm,amssymb}
\usepackage{graphicx}
\usepackage{enumerate}
\usepackage{amssymb,amsmath,graphicx,amsfonts,euscript}
\usepackage{color}
\usepackage{mathrsfs}
\usepackage{empheq}
\usepackage{cases}
\usepackage{cite}

\usepackage[colorlinks=true,backref]{hyperref}
\hypersetup{urlcolor=blue, citecolor=green, linkcolor=blue}

\usepackage[margin=3cm, a4paper]{geometry}

\def\H{{\cal H}}
\def\F{\mathscr{F} }
\def\N{\mathbb{N}}
\def\R{\mathbb{R}}
\def\Z{\mathbb{Z}}

\def\C{\mathbb{C}}

\def\H2{H^2(\R^N)}
\def\L2{L^2(\R^N)}
\def\to{\rightarrow}

\newcommand{\dt}{\,\mathrm{d}t}

\def\jb#1{\langle#1\rangle} \def\norm#1{\|#1\|}
\def\normb#1{\bigg\|#1\bigg\|} 

\def\H{{\cal H}}

\def\H1{H^1(\R)}

\newcommand{\wt}{\widetilde}

\newcommand{\al}{\alpha} 
 \newcommand{\de}{\delta}

 \newcommand{\Om}{\Omega}
\newcommand{\Ga}{\Gamma}

\newcommand{\im}{\mathop{\mathrm{Im}}}

   \newcommand{\I}{\infty}

 \newcommand{\Del}[1]{}

\numberwithin{equation}{section}

\newtheorem{thm}{Theorem}[section]

\newtheorem{lem}[thm]{Lemma}
\newtheorem{prop}[thm]{Proposition}
\newtheorem{definition}[thm]{Definition}
\theoremstyle{remark}
\newtheorem{remark}[thm]{Remark}
\newtheorem*{exam*}{Examples}

\newcommand{\EQ}[1]{\begin{align*}\begin{split} #1 \end{split}\end{align*}}
\newcommand{\EQn}[1]{\begin{align}\begin{split} #1 \end{split}\end{align}}
\newcommand{\EQnnsub}[1]{\begin{subequations}\begin{align} #1 \end{align}\end{subequations}}
\def\norm#1{\left\|#1\right\|}

\def\normb#1{\big\|#1\big\|}
\def\normbb#1{\Big\|#1\Big\|}

\def\absb#1{\big|#1\big|}

\def\brkb#1{\big(#1\big)}

\def\fbrk#1{\left\lbrace#1\right\rbrace}

\def\fbrkb#1{\big\lbrace#1\big\rbrace}

\def\jb#1{\langle#1\rangle}

\def\wt#1{\widetilde{#1}}
\def\wh#1{\widehat{#1}}
\def\wb#1{\overline{#1}}
\def\vec#1{\overrightarrow{#1}}

\def\pd{\partial}

\newcommand{\ra}{{\rightarrow}}

\def\lsm{\lesssim}
\def\gsm{\gtrsim}

\newcommand{\dd}{{\mathrm{d}}}

\def\al{\alpha}

\def\Om{\Omega}

\def\si{\sigma}
\def\de{\delta}

\begin{document}

\setcounter{page}{1}

\title[Wave operator]{Wave operator for the generalized derivative nonlinear Schr\"{o}dinger equation}

\author{Ruobing Bai}
\address{ School of Mathematics and Statistics\\
Henan University\\
Kaifeng 475004, China}
\email{baimaths@hotmail.com}
\thanks{}

\author{Jia Shen}
\address{Center for Applied Mathematics\\
Tianjin University\\
Tianjin 300072, China}
\email{shenjia@tju.edu.cn}
\thanks{}




\begin{abstract}\noindent
In this work, we prove the existence of wave operator for the following generalized derivative nonlinear Schr\"odinger equation
\begin{align*}
i\partial_t u+\partial_x^2 u +i |u|^{2\sigma}\partial_x u=0,
\end{align*}
with $(t,x)\in\R\times\R$,  $\sigma\in \mathbb{N}$, and $\si\geq 3$.
The study of wave operators is an important part of the scattering theory, and it is useful in the construction of the nonlinear profile and the large data scattering.

The previous argument for small data scattering in \cite{BaiWuXue-JDE}, which is based on the local smoothing effect and maximal function estimates, breaks down when considering the final data problem. The main reason is that the resolution space does not provide smallness near the infinite time. We overcome this difficulty by invoking the gauge transformation and the normal form method.
\end{abstract}

\maketitle

\section{Introduction}
\vskip 0.2cm

In this paper, we consider the existence of the wave operators for the following generalized derivative nonlinear Schr\"odinger equation
\begin{align}\label{GDNLS}\tag{gDNLS}
i\partial_t u+\partial_x^2 u +i |u|^{2\sigma}\partial_x u=0,
\end{align}
with the asymptotic condition
\EQ{
\|u(t,x)-e^{it\pd_x^2 }V_\pm\|_{H_x^1(\R)}\rightarrow 0, \quad \mbox{as }t\rightarrow\pm \infty.
}
Here $\sigma > 0$, $u(t,x):\R \times\R \to \C$ is the unknown function, and $V_\pm\in H_x^1(\R)$. The \eqref{GDNLS} equation describes the physical phenomenon of Alfv\'en waves with small but finite amplitude propagating along the magnetic field in cold plasmas (see for example \cite{MioOginoMinamiTakeda-JPS-76}).

The solution to (\ref{GDNLS}) equation is invariant under the scaling
\begin{equation}\label{eqs:scaling-p}
u(t,x)\to u_\lambda(t,x) = \lambda^{\frac1{2\sigma}} u(\lambda^2 t, \lambda x) \ \ {\rm for}\ \ \lambda>0.
\end{equation}
Denote
$$
s_c:=\frac 12-\frac1{2\sigma},
$$
then the scaling  leaves  $\dot{H}^{s_{c}}$ norm invariant, that is,
\begin{eqnarray*}
\|u(0)\|_{\dot H^{s_{c}}_x}=\|u_{\lambda}(0)\|_{\dot H^{s_{c}}_x}.
\end{eqnarray*}

When $\sigma=1$, take the gauge transformation
$$w(t,x) = u(t,x) \exp\Big(\frac i2 \int_{-\infty}^x |u(t,y)|^2\,\mathrm dy \Big),$$
then the equation in \eqref{GDNLS}
is transformed into the standard derivative nonlinear Schr\"odinger equation
\begin{align}\label{DNLS}\tag{DNLS}
i\partial_t w+\partial_x^2 w +i\partial_x(|w|^2w) =0.
\end{align}

Many researchers have studied the well-posedness theory for the equation \eqref{DNLS}. In terms of local well-posedness theory, Hayashi and Ozawa \cite{HaOza-PhyD-92, HaOza-SJMA-94} established local well-posedness in the space $H^1(\R)$ (see also the previous works \cite{GuoTan-PRSE-91, TuFu-80-DNLS}). Subsequently, Takaoka \cite{Ta-99-DNLS-LWP} proved the local well-posedness in $H^{\frac 12}(\mathbb R)$, while Biagioni and Linares \cite{BiLi-01-Illposed-DNLS-BO} showed that the ill-posedness in $H^s(\mathbb R)$ with $s<\frac 12$. Additional research results on local well-posedness are documented in \cite{MoYo, DanNiLi, Gr-05, Grhe-95, Herr, Ta-16-DNLS-LWP,DNY21CMP, Guo-Ren-Wang-DNLS} and the references therein.

Turning to global well-posedness theory for the equation \eqref{DNLS}, Hayashi and Ozawa \cite{HaOza-PhyD-92} established the result in $H^1(\R)$ under the condition that the initial data satisfies $\|u_0\|_{L^2}< \sqrt {2\pi}$. Later, Wu \cite{Wu-APDE-13, Wu-APDE-15} relaxed the size of initial data to $\|u_0\|_{L^2}< 2\sqrt \pi$, and then the regularity was improved to $H^{\frac 12}(\R)$ by Guo and Wu \cite{GuoWu-DCDS-17}. More recently, Bahouri and Perelman \cite{Bahouri-Perelman} removed the size restriction on initial data and proved the global well-posedness in $H^{\frac 12}(\R)$. Furthermore, Griffiths, Killip, Ntekoume, and Visan \cite{Gri-Killip-Visan, Gri-Killip-Nte-Visan} improved the regularity to the critical space $L^{2}(\R)$.
For the global well-posedness results in different settings, the readers can refer to \cite{CollianderTao-SJMA-01, CollianderTao-SJMA-02, Miao-Wu-Xu:2011:DNLS, Ta-01-DNLS-GWP, MoOh, Mosincat} and the corresponding references therein. Additional results related to stability and inverse scattering theory for \eqref{DNLS} can be found in \cite{ChSiSu-DNLS, JeLiPeSu, CoOh-06-DNLS, Fu-16-DNLS, GuWu95, guo-DNLS, GuoNingWu-Pre, JeLiPeSu-1, KwonWu-JAM-18, Stefan-W-15-MultiSoliton-DNLS, LiBing-NingCui-PRE, LiSiSu1, LiPeSu1, LiPeSu3, LiSiSu, Miao-Tang-Xu:2016:DNLS, Miao-Tang-Xu:2017:DNLS, PeSh-DNLS, PeSh-DNLS-2, TaXu-17-DNLS-Stability}, as well as the respective references.

The \eqref{GDNLS} equation is a generalization of \eqref{DNLS}. Hao \cite{Hao-CPAA-07} was the first to establish its local well-posedness in $H^\frac 12(\R)$ for $\sigma \ge \frac 52$. Subsequently, Santos \cite{Santos-JDE-2015} proved local well-posedness in the space $L^{\infty}((0, T); H^{\frac 32}(\R)\cap \langle x\rangle^{-1}H^{\frac 12}(\R))$ for small initial data when $\frac 12 \leq \sigma< 1$, and in $H^\frac 12(\R)$ for small initial data when $\sigma >1$. Recently, Linares, Ponce, and Santos  \cite{LinaresPonceSantos-PRE, LinaresPonceSantos-PRE-II} considered this equation with rough nonlinearity $0<\sigma<\frac 12$ and proved local well-posedness in an appropriately weighted Sobolev space when the initial data satisfies
\begin{align*}
{\mbox{Inf}}_{x\in \R}\langle x\rangle^m|u_0(x)|\geq \lambda >0, \quad m\in \Z^+.
\end{align*}

Additionally, Hayashi and Ozawa \cite{HaOza-JDE-16} also contributed to the understanding of local well-posedness for \eqref{GDNLS} in an open interval with a Dirichlet condition by showing it holds in $H^2$ for $\frac 12 \leq \sigma< 1$ and in $H^1$ for $\sigma >1$.

Next, we review the global well-posedness results for \eqref{GDNLS}. When $0<\sigma <1$, Hayashi and Ozawa \cite{HaOza-JDE-16} obtained global solutions in $H^1(\R)$ but without uniqueness. When $\sigma > 1$, Fukaya, Hayashi, and Inui \cite{FukayaHayashiInui-APDE-17} established global well-posedness in a class of invariant manifolds in $H^1(\R)$. Moreover, for the range $\frac{\sqrt{3}}{2}<\sigma<1$, Pineau and Taylor \cite{Pineau-Taylor} proved global well-posedness in $H^s(\R)$ with $1\leq s<4\sigma$.

Regarding the scattering theory for \eqref{GDNLS}, the only result is given by the first author, Wu, and Xue in \cite{BaiWuXue-JDE}, who proved the small data scattering in $H^{s}(\R)$ with $s\geq \frac 12$ when $\sigma \geq 2$. They also established that small data scattering does not hold when $\sigma <2$, thus providing an optimal small data scattering result for derivative Schr\"odinger type equations. Moreover, for \eqref{DNLS}, the modified wave operator was first given by Hayashi-Ozawa \cite{HaOza-MA-94}, and the modified scattering was proved by Guo, Hayashi, Lin, and Naumkin \cite{GHLN13SIAM}.

However, for the \eqref{GDNLS} equation, there is currently no result on the existence of wave operators. It is even surprise for us that the existence of wave operator is much harder than the small data scattering, the reasons can be found in the next subsection. Moreover, the wave operator is very important for the construction of the nonlinear profile and the large data scattering. In this paper, we aim at addressing the remaining problem.

\subsection{Main results}
To begin with, we define the wave operators for \eqref{GDNLS}.
\begin{definition}[Existence of the wave operators\label{Def1}]
Let $V_\pm\in H_x^1$. We say that the wave operator exists for \eqref{GDNLS} with the final data $V_+$ or $V_-$, if we can find some $T>0$, such that there exists a unique solution $u(t)\in C([T, +\infty); H_x^1(\R))$ or $u(t)\in C((-\infty, -T]; H_x^1(\R))$ to the equation \eqref{GDNLS}, satisfying
$$
\lim_{t\rightarrow  +\infty}\|u(t)-e^{it\pd_x^2}V_+\|_{H_x^1}=0\text{, or }\lim_{t\rightarrow  -\infty}\|u(t)-e^{it\pd_x^2}V_-\|_{H_x^1}=0.
$$
Moreover, we denote the wave operator as
\EQ{
	\Omega_{\pm}: V_\pm\rightarrow u(t).
}
\end{definition}

The following is our main result:
\begin{thm}\label{main theorem1}
Let $\sigma=k\in \mathbb{N}$ and $k\geq 3$. Suppose also that $V_\pm\in H_x^1(\R)$. Then, the wave operator exists for \eqref{GDNLS} with the final data $V_\pm$.
\end{thm}

We make several remarks regarding the result.
 \begin{remark}\label{r-1}
\begin{enumerate}
\item
This is the first result on the existence of wave operator for \eqref{GDNLS}.

\item
The study of wave operator has some new difficulties comparing to the scattering presented in \cite{BaiWuXue-JDE}. The resolution space used in \cite{BaiWuXue-JDE} does not provide smallness when considering the final data problem. We overcome this difficulty by employing the gauge transformation and normal form method.

\item
Our method is restricted to deal with the algebraic nonlinear terms. Moreover, the condition $k\geq 3$ arises from the integral term produced by the gauge transformation. However, we are of the opinion that this restriction is not sharp.

\item
The advantage of our approach is that the work space allows perturbation. Therefore, we can give the perturbed local theory, stability, and the existence of nonlinear profile for \eqref{GDNLS} in the energy space using the method in this paper. This may enable us to further study the large data scattering problem.

\end{enumerate}
\end{remark}

\subsection{The main difficulty and our methods}
\

$\bullet$ \textit{Main difficulty.}
Now, we explain the main difficulty in the investigation of wave operator problem. Following the approach by Bai-Xue-Wu \cite{BaiWuXue-JDE} to prove the small data scattering, we encounter the following
\EQ{
\normbb{\pd_x \int_I e^{i(t-s)\pd_x^2}(|u|^{2\si}u_x) \dd s}_{L_t^\I L_x^2} \lsm \normbb{ \int_I e^{i(t-s)\pd_x^2}(|u|^{2\si}\pd_x^2u) \dd s}_{L_t^\I L_x^2} + \text{ easier terms}.
}
The main task is to close the estimate with $H^1$-data. Note that there is a first-order derivative gap between $\pd_x^2u$ and the initial data. The only way to fill in the gap is the use of the smoothing effect estimate, which can absorb a $\frac12$-order derivative. Therefore, combining the homogeneous and inhomogeneous smoothing effect estimates,
\EQ{
\normbb{\pd_x \int_I e^{i(t-s)\pd_x^2}(|u|^{2\si}u_x) \dd s}_{L_t^\I L_x^2} \lsm & \normb{D^{\frac12} (|u|^{2\si}u_x)}_{L_x^1 L_t^2} \\
\lsm & \normb{D^{\frac32}u}_{L_x^\I L_t^2} \norm{u}_{L_x^{2\si} L_t^\I}^{2\si} +\text{ other terms}.
}
The solution can be small in those norms by taking small data for Cauchy problem. Unfortunately, the method cannot be applied to the final data problem for two reasons:
\begin{itemize}
\item
Those norms in the resolution space cannot provide the smallness near the infinite time, due to the \textit{endpoint} space-time exponents in $L_x^\I$ and $L_t^\I$.
\item
From the perspective of the smoothing effect estimate, in the above mentioned approach, the \textit{only choice} for the $D^{\frac12} (|u|^{2\si}u_x)$ and $D^{\frac32}u$ seems $L_x^1 L_t^2$ and $L_x^\I L_t^2$, respectively. Then, it seems difficult to close the estimates by perturbing the time-space index.
\end{itemize}
Consequently, it requires us to find another method to cover the first-order derivative gap when considering the final data problem.

$\bullet$ \textit{The gauge transformation.} It is widely recognized that the classical \eqref{DNLS} equation, when subjected to a gauge transformation, has seen significant advancements in studies related to its well-posedness due to alterations in its equation structure. The main observation is that the nonlinear term $|u|^{2\sigma}\pd_x\wb u$ possesses better non-resonance property than the original one.

We use the gauge transformation in the following form:
\begin{align}\label{8242}
w(t,x)=e^{\frac i2\int_{-\infty}^x|u(t,y)|^{2\sigma} dy}u(t,x),
\end{align}
which transforms the \eqref{GDNLS} equation to
\begin{align}\label{GDNLS-w}
i\partial_t w+\partial_x^2 w=F(w),
\end{align}
where $F(w)$ is given by
\begin{align}\label{9141}
F(w):=\sigma(\sigma-1)w \mbox{Im}\int_{-\infty}^x(\partial_y w\bar{w})^2|w|^{2\sigma-4}dy-\frac{\sigma+1}{4}|w|^{4\sigma}w+i\sigma|w|^{2\sigma-2}w^2 \partial_x \wb{w}.
\end{align}
Similar transforms also appeared in \cite{Hao-CPAA-07,HaOza-JDE-16}. The key point of this transformation lies in its ability to eliminate the most troublesome interactions within the derivative nonlinearity $|u|^{2\sigma}\pd_x u$. Although it introduces a new integral term when compared to the classical \eqref{DNLS} equation, effective $L_x^\I$-control can still be maintained.

$\bullet$ \textit{The normal form method.} We consider the case when $\si=k$ is an integer. The most complex scenario is as follows:
\EQ{
\int_{+\infty}^{t}e^{i(t-s)\pd_x^2}(w_{low}^{2k}\partial_x^2 \wb{w}_{high})ds,
}
where ``high" and ``low" represents the size of frequency, and we do not distinguish the $w_{low}$ and $\wb{w}_{low}$.

First, we need to employ the normal form method introduced by Shatah \cite{Sha85CPAM} to cover the derivative loss. The above integral can be rewritten as
\EQ{
e^{it\pd_x^2}\F_\xi^{-1}\int_{+\infty}^{t} \int_{\xi=\sum_{j=1}^{2k+1}\eta_j}e^{-is\Phi}\wh{v}_{low}(\eta_1) \cdots \wh{v}_{low}(\eta_{2k}) \wh{\pd_x^2\wb{v}_{high}}(\eta_{2k+1})\dd \eta_1 \cdots \dd \eta_{2k} \dd \eta_{2k+1} \dd s,
}
where the profile $v:=e^{-it\pd_x^2}w$, and
the phase function
\EQ{
	\Phi:=\xi^2\pm\eta_{1}^2\pm\cdots\pm \eta_{2k}^2+\eta_{2k+1}^2.
}
We observe that this integral is temporal non-resonant in the sense that $\Phi \thicksim \eta_{2k+1}^2$. This is the main reason why we use the gauge transform to turn the nonlinearity $|w|^{2k}w_x$ into $|w|^{2k}\wb w_x$. After integrating by parts in $s$, we obtain several terms in the form of multilinear operator $T$ (called the normal form transform):
\EQ{
& T(w_{low},w_{low},\cdots,w_{low},w_{high}) \\
& := \F_\xi^{-1} \int \int_{\xi=\sum_{j=1}^{2k+1}\eta_j} \frac{1}{i\Phi} \wh{w}_{low}(\eta_1) \cdots \wh{w}_{low}(\eta_{2k}) \wh{\pd_x^2\wb{w}_{high}}(\eta_{2k+1})\dd \eta_1 \cdots \dd \eta_{2k} \dd \eta_{2k+1}.
}
By the Coifman-Meyer Theorem, the normal form can roughly be viewed as
\EQ{
T(w_{low},\cdots,w_{low},w_{high}) &\sim \jb{\nabla}^{-2} (w_{low}\cdots w_{low}\pd_{x}^2w_{high})\\
&\sim  w_{low}\cdots w_{low}w_{high},
}
where the $\jb{\nabla}^{-2}$ is derived due to the factor $1/\Phi$. Then, there is no derivative loss in $T$.

Unfortunately, the application of normal form method on the final data problem also has two more difficulties:
\begin{itemize}
\item
$T$ can not be defined for $t=\I$.
\item
The boundary term $T$ can not gain any small factor in the energy space near the infinite time.
\end{itemize}
We overcome the first problem by approximating the solution $w$ by a sequence of solution $w_n$ with Cauchy data at $t=t_n$, which tends to infinity. For the second problem, we modify the definition of normal form transform by inserting a high frequency cutoff:
\EQ{
\wt T(w_{low},w_{low},\cdots,w_{low},w_{high}) := \sum_{N\ge N_0} P_N T(w_{low},w_{low},\cdots,w_{low},w_{high}).
}
Using this new transform, we can obtain a small factor $N_0^{-\frac 12}$ from the regularity room by the Sobolev embedding, while the low frequency case when $N\leq N_0$ can be also bounded easily without the use of normal form transform. For the precise definition of normal form transform in this paper, please see the details in Definition \ref{defn:normal-form}. For the previous results on the application of normal form method on the scattering theory in energy space, we refer the readers to the scattering of the 3d radial Zakharov system by Guo, Nakanishi, and Wang \cite{Guo-Nakanishi-IMRN, Guo-Nakanishi-Wang-Adv}.

\subsection{Organization of the paper} The rest of the paper is organized as follows. In Section 2, we give some basic notations, lemmas, the gauge transform and the normal form transform that will be used in this paper. In Section 3, we give some necessary nonlinear estimates. In Section 4, we give the proof of Theorem \ref{main theorem1}.

\section{Preliminary}\label{sec:notations}

\subsection{Notations}
For any $a\in \mathbb{R}$, $a\pm:=a\pm\epsilon$ for arbitrary small $\epsilon>0$. For any $z\in \mathbb{C}$, we define $\mbox{Re}z$ and $\mbox{Im}z$ as the real and imaginary part of $z$, respectively. Moreover, $\langle \cdot\rangle =(1+|\cdot|^2)^{\frac 12}$. We write $X \lesssim Y$ or $Y \gtrsim X$ to indicate $X \leq CY$ for
some constant $C>0$. If $X \leq CY$ and $Y \leq CX$, we write $X\sim Y$. Throughout the whole paper, the letter $C$ will denote suitable positive constant that may vary from line to line. We define the Fourier transform of $f$ and Fourier inverse transform of $g$ by:
\begin{align*}
\hat{f}(\xi)=\mathscr{F}f(\xi):=\int_{\R^d}e^{-ix\cdot \xi}f(x)dx\text{, and }\mathscr{F}^{-1}g(x):=\frac{1}{(2\pi)^d}\int_{\R^d}e^{ix\cdot \xi}g(\xi)d\xi.
\end{align*}
We use the following norms to denote the mixed spaces $L_t^qL_x^r$, that is
\begin{align*}
\|u\|_{L_t^qL_x^r}=\big(\int \|u\|_{L_x^r}^qdt\big)^{\frac{1}{q}}.
\end{align*}
When $q=r$ we abbreviate $L_t^qL_x^q$ as $L_{t,x}^q$.

We also need the usual inhomogeneous Littlewood-Paley decomposition for the dyadic number. We take a cut-off function $\phi \in C_0^{\infty}(0, \infty)$ such that
$$
\phi(r)=\left\{ \aligned
    &1, r\leq 1,
    \\
    &0, r\geq 2.
   \endaligned
  \right.
$$
Then we introduce the spatial cut-off function. For dyadic number $N\in 2^{\mathbb{N}}$, when $N\geq 1$, let $\phi_{\leq N}(r)=\phi(N^{-1}r)$ and $\phi_{N}(r)=\phi_{\leq N}(r)-\phi_{\leq \frac N2}(r)$. We define the Littlewood-Paley dyadic operator
\begin{align*}
f_{\leq N}=P_{\leq N}f:=\mathscr{F}^{-1}(\phi_{\leq N}(|\xi|)\hat{f}(\xi))\text{, and }f_{ N}=P_{ N}f:=\mathscr{F}^{-1}(\phi_{ N}(|\xi|)\hat{f}(\xi)).
\end{align*}
we also define that $f_{\geq N}=P_{\geq N}f:=f-f_{\leq N}$, $f_{\ll N}=P_{\ll N}f$, $f_{\gtrsim N}:=P_{\gtrsim N}f$, $f_{\lesssim N}:=P_{\lesssim N}f$, and $P_{\thicksim N}=f_{ \thicksim N}$.

For any $1\leq p <\infty$, we define $l_{N}^p=l_{N\in 2^{\mathbb{N}}}^p$ by its norm
$$
\|c_N\|_{l_{N\in 2^{\mathbb{N}}}^p}^p:=\sum_{N\in 2^{\mathbb{N}}}|c_N|^p.
$$
For $p =\infty$, we define $l_{N}^{\infty}=l_{N\in 2^{\mathbb{N}}}^{\infty}$ by its norm
$$
\|c_N\|_{l_{N\in 2^{\mathbb{N}}}^{\infty}}:=\max_{N\in 2^{\mathbb{N}}}|c_N|.
$$
In this paper, we also use the following abbreviations
$$
\sum_{N:N\leq N_1}a_N:=\sum_{N\in 2^{\mathbb{N}}: N\leq N_1}a_N.
$$
\subsection{Basic lemmas}

\quad In this section, we state some preliminary estimates that will be used in our later sections. Firstly, we recall the well-known Strichartz's estimates, see \cite{Keel-Tao-1998} for instance.

\begin{lem}[Strichartz's estimates\label{lem:strichartz}]
Let $I\subset \R$ be an interval. For all admissible pair $(q_j,r_j),j=1,2,$  satisfying
\begin{align*}
2\leq q_j,r_j\leq \infty\quad and\quad \frac 2 q_j=\frac12-\frac 1r_j,
\end{align*}
the following estimates hold:
\begin{align}\label{Strichartz1}
\|e^{it\pd_x^2}f \|_{L_t^{q_j} L_x^{r_j}(I\times \R)} \lesssim  \| f  \|_{L^2(\R)};
\end{align}
and
\begin{align}\label{Strichartz2}
\Big\|\int_0^t e^{i(t-t')\pd_x^2} F(t',x)\dt'\Big\|_{L_t^{q_1} L_x^{r_1}(I\times \R)}
\lesssim\|F\|_{L_t^{q_2'} L_x^{r_2'}(I\times \R)},
\end{align}
where $\frac 1 {q_2}+\frac 1{q_2'}=\frac 1 {r_2}+\frac 1{r_2'}=1$.
\end{lem}

\begin{lem}[Bernstein estimates\label{lem:Bernstein}]
For $1\leq p \leq q \leq \infty$,
\begin{align*}
\||\nabla|^{\pm s}P_N f\|_{L_x^p(\R^d)}&\thicksim N^{\pm s}\|P_N f\|_{L_x^p(\R^d)},\\
\|P_N f\|_{L_x^q(\R^d)}&\lesssim  N^{\frac dp-\frac dq}\|P_N f\|_{L_x^p(\R^d)}.
\end{align*}
\end{lem}

\begin{lem}[Llittlewood-Paley estimates\label{lem:littlewood-Paley}]
For $1< p <\infty$ and $f\in L_x^p(\R^d)$. Then, we have
\begin{align*}
\|f_N\|_{L_x^pl_N^2}\thicksim_p\|f\|_{L_x^p}.
\end{align*}
\end{lem}

\begin{lem}[Multilinear Coifman-Meyer multiplier estimates, see \cite{Co-Me-91}\label{lem:Coifman-Meyer}]
Let the function $m$ on ${(\R^n)^k}$ is bounded and let $T_m$ be the corresponding m-linear multiplier operator on $\R^n (n\geq 1)$
\begin{align*}
T_m(f_1,\cdots, f_k)(x)=\int_{(\R^n)^k}m(\eta_1,\cdots,\eta_k)\hat{f_1}(\eta_1)\cdots\hat{f_1}(\eta_k)e^{ix(\eta_1+\cdots+\eta_k)}d\eta_1\cdots d\eta_k.
\end{align*}
 If $L$ is sufficiently large and $m$ satisfies
\begin{align*}
\Big|\partial_{\eta_1}^{\al_1}\cdots\partial_{\eta_k}^{\al_k}m(\eta_1,\cdots,\eta_k)\Big|\lesssim_{\al_1,\cdots, \al_k}(|\eta_1|+\cdots+|\eta_k|)^{-(|\al_1|+\cdots+|\al_k|)},
\end{align*}
for multi-indices $\al_1,\cdots, \al_k$ satisfying $|\al_1|+\cdots+|\al_k|\leq L$. Then, for $1< p, p_1, \cdots, p_k\leq \infty$ and $\frac 1p=\frac{1}{p_1}+\cdots+\frac{1}{p_k}$, we have
\begin{align*}
\|T_m(f_1,\cdots, f_k)\|_{L_x^p}\leq C\|f_1\|_{L_x^{p_1}}\cdots\|f_k\|_{L_x^{p_k}}.
\end{align*}
\end{lem}

\subsection{Gauge transform}

Let $u$ be the solution to \eqref{GDNLS}, namely
\EQ{
i\pd_t u +\pd_x^2 u +i|u|^{2\si} u_x=0,
}
with $\si>0$. Let $a\in\R$ and define
\EQn{\label{eq:gauge-transform-general}
w(t,x):=e^{-ia\int_{-\I}^x |u(t,y)|^{2\si} \dd y} u(t,x).
}
Then, we have that $w(t,x)$ solves
\EQn{\label{eq:gdnls-gauge-general}
i\pd_t w +\pd_x^2 w + & 2a\si(\si-1)w\cdot \im \int_{-\I}^x \brkb{|w|^{2\si-4}(\wb w \cdot\pd_y w)^2}(t,y) \dd y\\
+ & a(a\si-\frac12) |w|^{4\si}w+2a\si i|w|^{2\si-2} w^2 \pd_x\wb w + (2a+1)i|w|^{2\si} \pd_x w=0.
}

When $\si=1$, this formula coincides the classical gauge transform for \eqref{DNLS}. When $a=1/2$, then $w=e^{-\frac i2\int_{-\I}^x |u(y)|^{2\si} \dd y} u$ and
\EQ{
i\pd_t w + \pd_x^2 w + i\pd_x(|w|^2w) =0.
}
When $a=-1/2$, then $w=e^{\frac i2\int_{-\I}^x |u(y)|^{2\si} \dd y} u$ and
\EQ{
i\pd_t w + \pd_x^2 w - i|w|^2\pd_x\wb w + \frac12|w|^4w =0.
}

In this paper, we use the gauge transform in \eqref{eq:gauge-transform-general} with $a=-1/2$. This removes the troublesome term $|w|^{2\si} \pd_x w$. Therefore, the equation  \eqref{eq:gdnls-gauge-general} reads
\EQn{\label{eq:gdnls-gauge}
i\pd_t w +\pd_x^2 w = & \si(\si-1)w\cdot \im \int_{-\I}^x \brkb{|w|^{2\si-4}(\wb w \cdot\pd_y w)^2}(t,y) \dd y\\
& - \frac{\si+1}{4} |w|^{4\si}w + i \si |w|^{2\si-2} w^2 \pd_x\wb w.
}

\subsection{Normal form transform}
Next, we apply the normal form transform to the equation \eqref{eq:gdnls-gauge} with $\si=k$ and
\EQ{
k\in\N\text{, and }k\ge 3.
}
Denote that
\EQn{\label{eq:nonlinear-term-gauge}
F_1 (w) := & k(k-1)w\cdot \im \int_{-\I}^x \brkb{|w|^{2k-4}(\wb w \cdot\pd_y w)^2}(t,y) \dd y;\\
\quad F_2(w):=& - \frac{k+1}{4} |w|^{4k}w;\\
F_3(w) :=& i k|w|^{2k-2} w^2 \pd_x\wb w.
}
Now, we consider the Cauchy problem
\EQn{
\label{eq:gdnls-gauge-cauchy}
\left\{ \begin{aligned}
&i\pd_t w + \pd_x^2 w = F_1(w) + F_2(w) + F_3(w), \\
& w(t_0) = w_0.
\end{aligned}
\right.
}
Then, the integral equation for \eqref{eq:gdnls-gauge-cauchy} is
\EQn{\label{eq:gdnls-gauge-integral}
w(t,x) = e^{i(t-t_0)\pd_x^2} w_0(x) - i\int_{t_0}^t e^{i(t-s)\pd_x^2}(F_1+F_2+F_3)(w(s,x)) \dd s.
}

Next, we define the normal form transform with respect to some suitable frequency part of the nonlinear term $F_3$.
\begin{definition}[Normal form transform]
\label{defn:normal-form}
Let $k\in\N$ with $k\ge3$ and $\xi,\eta_j\in\R$ for $j=1,2,3,\cdots,2k+1$. Denote that $\vec \eta=(\eta_1,\eta_2,\eta_3,\cdots,\eta_{2k+1})\in\R^{2k+1}$. Let $\iota_j\in\fbrk{+1,-1}$ for $j=1,2,3,\cdots,2k+1$ such that
\EQ{
\iota_{2n-1} := +1\text{, } \iota_{2n} :=-1 \text{, for }n=1,2,\cdots,k-1,
}
and
\EQ{
\iota_{2k-1}=\iota_{2k}:=+1\text{, }\iota_{2k+1}:=-1.
}
Define the hypersurface
\EQ{
\Ga_{\xi}(\vec \eta) := \fbrkb{\vec \eta\in\R^{2k+1}:\xi=\sum_{j=1}^{2k+1}\eta_j}.
}
The phase function is defined by
\EQ{
\Phi(\vec\eta):=(\sum_{j=1}^{2k+1}\eta_j)^2- \sum_{j=1}^{2k+1}\iota_j \eta_j^2.
}
Denote the multiplier
\EQ{
m_N(\vec\eta):=\frac{\eta_{2k+1}}{\Phi(\vec\eta)} \cdot \prod_{j=1}^{2k} \phi_{\ll N}(|\eta_j|) \cdot\phi_N(|\eta_{2k+1}|).
}
For any function $f$, we define
\EQ{
f^{(+1)} := f\text{, and }f^{(-1)} := \wb f.
}
Using the above notations, we give the following definitions:
\begin{enumerate}
\item (Normal form transform)
Let $N_0\in 2^\N$. Now, we define the normal form transform for functions $f_1,f_2,\cdots,f_{2k+1}$ by
\EQ{
\Om(f_1,f_2,\cdots,f_{2k+1})(x) := \sum_{N:N\ge N_0} \int_{\Ga_\xi(\vec\eta)} e^{ix\brkb{\sum_{j=1}^{2k+1}\eta_j}} m_N(\vec \eta) \prod_{j=1}^{2k+1} \wh{f_j^{(\iota_j)}}(\eta_j) \dd \vec\eta.
}
\item (Simplified notation for normal form)
If $f_1=f_2=\cdots=f_{2k+1}$, we denote that
\EQ{
	\Om(f):=\Om(f_1,f_2,\cdots,f_{2k+1})=\Om(f,f,\cdots,f).
}
If $f_j=f$ for all $1\le j\le 2k+1$ except for $j=l$, and $f_l=(-1)^lg$, then we denote that
\EQ{
	\Om_l(f,g):=\Om(f_1,f_2,\cdots,f_{2k+1}).
}
\item (Resonance term)
Finally, we also define the resonance part of the nonlinear term by
\EQ{
R(w):= & \sum_{N\in2^\N}(|w|^{2k-2} w^2 - |w_{\ll N}|^{2k-2} w_{\ll N}^2)\pd_x\wb w_N \\
& + \sum_{N:N\le N_0} |w_{\ll N}|^{2k-2} w_{\ll N}^2\pd_x\wb w_N.
}
\end{enumerate}

\end{definition}
\begin{remark}
We can check that $Nm_N$ satisfies the conditions of Coifman-Meyer's multiplier in Lemma \ref{lem:Coifman-Meyer}. Particularly, on the support of $m_N$, we have
\EQ{
\Phi(\vec\eta) \sim N^{2}\text{, and }|m_N(\vec\eta)| \sim N^{-1} \prod_{j=1}^{2k} \phi_{\ll N}(|\eta_j|) \cdot\phi_N(|\eta_{2k+1}|).
}
\end{remark}
Using the notations in Definition \ref{defn:normal-form}, we can transform the equation \eqref{eq:gdnls-gauge-integral} into
\EQn{\label{eq:gdnls-normal-form}
w(t,x) = & e^{i(t-t_0)\pd_x^2} w_0(x) - i\int_{t_0}^t e^{i(t-s)\pd_x^2}(F_1+F_2)(w(s,x)) \dd s  \\
& - ke^{i(t-t_0)\pd_x^2} \Om(w_0(x)) + k\Om(w(t,x))+ k\int_{t_0}^t e^{i(t-s)\pd_x^2} R(w(s,x)) \dd s \\
& -ik\sum_{j=1}^{2k+1}\int_{t_0}^t e^{i(t-s)\pd_x^2}\Om_j(w,(F_1+F_2+F_3)(w))(s,x) \dd s.
}
This formula can be derived by integrating in $s$ after transforming the equation \eqref{eq:gdnls-gauge-integral} into Fourier space. The procedure is classical, so we omit the details.

\section{Nonlinear Estimates}

We first define the auxiliary spaces $X(I)$ and $Y(I)$ for $I\subset\R$ by the following norms,
\begin{align}
\|w\|_{X(I)}:=\|w\|_{L_t^{\infty}H_x^1(I\times \R)} + \|\langle \partial_x\rangle w\|_{L_t^{4}L_x^{\infty}(I\times \R)},
\end{align}
and
\begin{align}
\|w\|_{Y(I)}:=\|\langle \partial_x\rangle w\|_{L_{t,x}^{6}(I\times \R)}.
\end{align}
In the proofs of the following lemmas, we always restrict the variables on $(t,x)\in I\times\R$. We also omit the dependence on $k$, namely writing $C$ for $C(k)$.
\subsection{Easier nonlinear terms}
\begin{lem}[Estimates for $F_1$ and $F_2$]\label{lem:nonlinear-estimate-f1f2}
Let $I \subset\R$ be an interval containing $t_0$.
Then,
\EQn{\label{eq:nonlinear-estimate-f1}
\normbb{\int_{t_0}^t e^{i(t-s)\pd_x^2}F_1(w(s)) \dd s}_{X\cap Y(I)} \lsm \|w\|_{X(I)}^{2k-5}\|w\|_{Y(I)}^6+\|w\|_{X(I)}^{2k-2}\|w\|_{Y(I)}^3,
}
and
\EQn{\label{eq:nonlinear-estimate-f2}
\normbb{\int_{t_0}^t e^{i(t-s)\pd_x^2}F_2(w(s)) \dd s}_{X\cap Y(I)} \lsm \|w\|_{X(I)}^{4k-2}\|w\|_{Y(I)}^3.
}
\end{lem}
\begin{proof}
We first prove \eqref{eq:nonlinear-estimate-f1}.
By Strichartz's estimates, we obtain
\EQnnsub{
\normbb{\int_{t_0}^t e^{i(t-s)\pd_x^2}F_1(w(s)) \dd s}_{X\cap Y(I)} \lsm & \norm{\jb{\pd_x}F_1(w)}_{L_t^1 L_x^2} \nonumber\\
\lsm &  \Big\|w {\rm{Im}}\int_{-\infty}^x(\partial_y w\wb{w})^2|w|^{2k-4}dy\Big\|_{L_t^1 L_x^2} \label{est:nonlinear-estimate-f1-1}\\
& + \Big\|\pd_x\brkb{ w {\rm{Im}}\int_{-\infty}^x(\partial_y w\wb{w})^2|w|^{2k-4}dy }\Big\|_{L_t^1 L_x^2}. \label{est:nonlinear-estimate-f1-2}
}
To prove this, we need the $L_x^\I$-estimate for the spatial integral term. By H\"{o}lder's inequality and Sobolev embedding,
\EQn{\label{8167}
	\Big\|\int_{-\infty}^x(\partial_y w\wb{w})^2|w|^{2k-4}dy\Big\|_{L_t^1L_x^{\infty}}
	\lesssim&\|(\partial_x w\wb{w})^2|w|^{2k-4}\|_{L_{t,x}^1}\\
	\lesssim& \|\partial_x w \|_{L_{t,x}^6}^2\|w \|_{L_{t,x}^6}^4\|w\|_{L_{t,x}^{\infty}}^{2k-6}\\
	\lesssim&\|\partial_x w \|_{L_{t,x}^6}^2\|w \|_{L_{t,x}^6}^4\|w\|_{L_t^{\infty}H_x^1}^{2k-6}\\
	\lesssim& \|w\|_{X(I)}^{2k-6} \|w\|_{Y(I)}^6.
}

Now, we estimate the term \eqref{est:nonlinear-estimate-f1-1}. By H\"older's inequality and \eqref{8167},
\begin{align*}
\eqref{est:nonlinear-estimate-f1-1} \lesssim& \Big\|w \int_{-\infty}^x(\partial_y w\wb{w})^2|w|^{2k-4}dy\Big\|_{L_t^1L_x^2}\\
\lesssim& \|w\|_{L_t^{\infty}L_x^2}\Big\|\int_{-\infty}^x(\partial_y w\wb{w})^2|w|^{2k-4}dy\Big\|_{L_t^1L_x^{\infty}}\\
	\lesssim& \|w\|_{X(I)}^{2k-5}\|w\|_{Y(I)}^6.
\end{align*}
For \eqref{est:nonlinear-estimate-f1-2}, by Sobolev embedding, \eqref{8167}, and H\"older's inequality,
\begin{align*}
\eqref{est:nonlinear-estimate-f1-2} \lesssim& \Big\|\partial_xw \int_{-\infty}^x(\partial_y w\wb{w})^2|w|^{2k-4}dy\Big\|_{L_t^1L_x^2}+\Big\|w (\partial_x w\wb{w})^2|w|^{2k-4}\Big\|_{L_t^1L_x^2}\\
	\lesssim& \|\partial_xw\|_{L_t^{\infty}L_x^2}\Big\|\int_{-\infty}^x(\partial_y w\wb{w})^2|w|^{2k-4}dy\Big\|_{L_t^1L_x^{\infty}}+\|\partial_xw\|_{L_t^{4}L_x^{\infty}}^2\|w \|_{L_{t,x}^6}^3\|w \|_{L_{t,x}^{\infty}}^{2k-4}\\
	\lesssim& \|w\|_{X(I)}^{2k-5}\|w\|_{Y(I)}^6+\|w\|_{X(I)}^{2k-2}\|w\|_{Y(I)}^3.
\end{align*}
Hence, \eqref{eq:nonlinear-estimate-f1} follows by the above two estimates.

For the $F_2$-term, applying Strichartz's estimates, H\"{o}lder's inequality, and Sobolev embedding,
\EQ{
\normbb{\int_{t_0}^t e^{i(t-s)\pd_x^2}F_2(w(s)) \dd s}_{X\cap Y(I)} \lesssim& \big\|\langle\partial_x\rangle(|w|^{4k}w)\big\|_{L_t^1L_x^2}\\
\lesssim& \|w\|_{L_{t,x}^{\infty}}^{4k-4}\| \langle\partial_x\rangle w \|_{L_{t,x}^6}\| w \|_{L_{t,x}^6}^2\|w \|_{L_t^4L_x^{\infty}}^2\\
\lesssim& \|w\|_{X(I)}^{4k-2}\|w\|_{Y(I)}^3.
}
This gives \eqref{eq:nonlinear-estimate-f2}.
\end{proof}

\subsection{Boundary terms}
\begin{lem}[Boundary terms]\label{lem:nonlinear-estimate-boundary}
Let $I \subset\R$ be an interval containing $t_0$. Then, for any $N_0\in2^\N$,
\EQ{
& \normb{e^{i(t-t_0)\pd_x^2} \Om(w_0(x))}_{X\cap Y(I)} \lsm N_0^{-1} \norm{w_0}_{L_x^{\I}}^{2k} \norm{w_0}_{H_x^1}, \\
& \norm{\Om(w)}_{X(I)} \lsm N_0^{-\frac{1}{2}-}\|w\|_{X(I)}^{2k+1}, \\
& \norm{\Om(w)}_{Y(I)} \lsm N_0^{-1} \norm{w}_{X(I)}^{2k} \norm{w}_{Y(I)}.
}
\end{lem}
\begin{proof}
Note that heuristically,
\EQ{
\Om(w) \sim \sum_{N:N\ge N_0}\jb{\pd_x}^{-1}(w_{\ll N}w_{\ll N}\cdots w_{\ll N} w_{N}).
}
By the Strichartz's estimates and Lemma \ref{lem:Coifman-Meyer},
\EQ{
\normb{e^{i(t-t_0)\pd_x^2} \Om(w_0(x))}_{X\cap Y(I)} \lsm & \norm{\Om(w_0(x))}_{H_x^1} \\
\lsm & \sum_{N:N\ge N_0} \norm{P_{\ll N}w_0}_{L_x^\I}^{2k} \norm{P_Nw_0}_{L_x^2} \\
\lsm & N_0^{-1} \norm{w_0}_{L_x^{\I}}^{2k} \norm{w_0}_{H^1}.
}
Applying Lemmas \ref{lem:Coifman-Meyer} and \ref{lem:Bernstein}, and Sobolev embedding, we get
\begin{align*}
\norm{\Om(w)}_{X(I)} \lesssim& \sum_{N:N\geq N_0}\|w_{\ll N}\|_{L_{t,x}^{\infty}}^{2k}\|w_{N}\|_{L_{t}^{\infty}L_x^{2}}+ \sum_{N:N\geq N_0}\|w_{\ll N}\|_{L_{t,x}^{\infty}}^{2k-1}\|w_{\ll N}\|_{L_t^{4}L_x^{\infty}}\|w_{N}\|_{L_{t,x}^{\infty}} \nonumber\\
\lesssim& \sum_{N:N\geq N_0}N^{-1}\|w_{\ll N}\|_{L_t^{\infty}H_x^{1}}^{2k}\|w_{N}\|_{L_t^{\infty}H_x^{1}}\\
&+\sum_{N:N\geq N_0}N^{-\frac{1}{2}-}\|w_{\ll N}\|_{L_{t,x}^{\infty}}^{2k-1}\|w_{\ll N}\|_{L_t^{4}L_x^{\infty}}\|w_{N}\|_{L_t^{\infty}H_x^{1}}\nonumber \\
	\lesssim& N_0^{-1}\|w\|_{L_t^{\infty}H_x^{1}}^{2k}\|w_{}\|_{L_t^{\infty}H_x^{1}}+N_0^{-\frac{1}{2}-}\|w\|_{L_t^{4}L_x^{\infty}}\|w\|_{L_t^{\infty}H_x^{1}}^{2k}\\
	\lesssim&N_0^{-\frac{1}{2}-}\|w\|_{X(I)}^{2k+1}.
\end{align*}
Similarly,
\EQ{
\norm{\Om(w)}_{Y(I)} \lesssim& \sum_{N:N\geq N_0}\|w_{\ll N}\|_{L_{t,x}^{\infty}}^{2k}\|w_{N}\|_{L_{t,x}^{6}}\nonumber\\
\lsm & \sum_{N:N\geq N_0}N^{-1}\|w_{\ll N}\|_{L_{t,x}^{\infty}}^{2k}\|\partial_xw_{N}\|_{L_{t,x}^{6}}\\
\lsm & N_0^{-1}\|w\|_{L_t^{\infty}H_x^1}^{2k}\|\partial_xw\|_{L_{t,x}^{6}}\nonumber\\
\lsm & N_0^{-1} \norm{w}_{X(I)}^{2k} \norm{w}_{Y(I)}.
}
This finishes the proof.
\end{proof}
\subsection{Resonance terms}
\begin{lem}[Resonance terms]\label{lem:nonlinear-estimate-resonance}
Let $I \subset\R$ be an interval containing $t_0$. Then, for any $N_0\in2^\N$,
\EQ{
\normbb{\int_{t_0}^t e^{i(t-s)\pd_x^2}R(w(s)) \dd s}_{X\cap Y(I)} \lsm N_0^2 \norm{w}_{X(I)}^{2k-2}\norm{w}_{Y(I)}^{3}.
}
\end{lem}
\begin{proof}
Recall $R$ in Definition \ref{defn:normal-form} and use the Strichartz's estimates, then
\EQnnsub{
\normbb{\int_{t_0}^t e^{i(t-s)\pd_x^2}R(w(s)) \dd s}_{X\cap Y(I)} \lsm & \normbb{\jb{\pd_x}\sum_{N\in2^\N}(|w|^{2k-2} w^2 - |w_{\ll N}|^{2k-2} w_{\ll N}^2)\pd_x\wb w_N}_{L_t^1 L_x^2} \label{est:nonlinear-estimate-resonance-1}\\
& + \normbb{\jb{\pd_x}\sum_{N:N\le N_0} |w_{\ll N}|^{2k-2} w_{\ll N}^2\pd_x\wb w_N}_{X\cap Y(I)}. \label{est:nonlinear-estimate-resonance-2}
}

Now, we deal with \eqref{est:nonlinear-estimate-resonance-1}. For this term, it's easy to transfer the derivative to the high frequency one. We also do not distinguish between $w$ and $\wb w$. First, we denote a set of dyadic number
\EQ{
& \Psi_{N_{2k}} := \fbrk{(N_1,N_2,\cdots,N_{2k-1})\in (2^\N)^{2k-1} : N_1\le N_2\le \cdots \le N_{2k-1}\le N_{2k}}.
}
Then, by Littlewood-Paley decomposition and renumbering the functions according to their frequency, we have that
\EQ{
\absb{\sum_{N\in2^\N}(|w|^{2k-2} w^2 - |w_{\ll N}|^{2k-2} w_{\ll N}^2)\pd_x\wb w_N} \lsm & \absb{ \sum_{N\in2^\N}\sum_{N_{2k}:N\lsm N_{2k}}  A_{N_{2k}}\partial_x \wb{w}_N} \\
\sim & |\sum_{N_{2k}\in2^\N}  A_{N_{2k}}\partial_x \wb{w}_{\lsm N_{2k}}|,
}
where
\EQ{
A_{N_{2k}} := \sum_{(N_1,N_2,\cdots,N_{2k-1})\in \Psi_{N_{2k}}}w_{N_1}w_{N_2}\cdots w_{N_{2k-1}}w_{N_{2k}}.
}
Therefore, by the above notation, using the Littlewood-Paley theory and H\"older's inequality in $N_{2k}$,
\EQ{
\eqref{est:nonlinear-estimate-resonance-1} \lsm & \normbb{\jb{\pd_x} \sum_{N_{2k}\in2^\N}  A_{N_{2k}}\partial_x \wb{w}_{\lsm N_{2k}}}_{L_t^1 L_x^2} \\
\lsm & \normbb{M P_M\brkb{\sum_{N_{2k}:N_{2k}\gsm M}  A_{N_{2k}}\partial_x \wb{w}_{\lsm N_{2k}}}}_{L_t^1 L_x^2 l_{M}^2} \\
\lsm & \normbb{ \sum_{N_{2k}:N_{2k}\gsm M} \frac{M}{N_{2k}} N_{2k} P_M\brkb{A_{N_{2k}}\partial_x \wb{w}_{\lsm N_{2k}}}}_{L_t^1 L_x^2 l_{M}^2} \\
\lsm & \normbb{ N_{2k} P_M\brkb{A_{N_{2k}}\partial_x \wb{w}_{\lsm N_{2k}}}}_{L_t^1 L_x^2 l_{M}^2 l_{N_{2k}}^2} \\
\lsm & \normbb{N_{2k}A_{N_{2k}}\partial_x \wb{w}_{\lsm N_{2k}}}_{L_t^1 L_x^2 l_{N_{2k}}^2}.
}
Moreover, by H\"{o}lder's inequality in $l_{N_{2k-1}}^2$,
\EQ{
|A_{N_{2k}}| \lsm & \sum_{(N_1,N_2,\cdots,N_{2k-1})\in \Psi_{N_{2k}}}w_{N_1}w_{N_2}\cdots N_{2k-1}^{-1} N_{2k-1}w_{N_{2k-1}}w_{N_{2k}} \\
\lsm & \prod_{i=1}^{2k-2} \norm{w_{N_i}}_{l_{N_i}^\I}\cdot \sum_{N_1\le N_2\le\cdots\le N_{2k-2}}\brkb{\sum_{N_{2k-1}:N_{2k-1}\ge N_{2k-2}} N_{2k-1}^{-2}}^{1/2}  \cdot \norm{N_{2k-1}w_{N_{2k-1}}}_{l_{N_{2k-1}}^2} |w_{N_{2k}}| \\
\lsm &  \norm{w_{N}}_{l_{N}^\I}^{2k-2}\cdot \sum_{N_1\le N_2\le\cdots\le N_{2k-2}} N_{2k-2}^{-1}  \cdot \norm{Nw_{N}}_{l_{N}^2} |w_{N_{2k}}| \\
\lsm & \norm{w_{N}}_{l_{N}^\I}^{2k-2} \norm{Nw_{N}}_{l_{N}^2} |w_{N_{2k}}|.
}
Then, combing the above two estimates, by H\"older's inequality,
\EQ{
\eqref{est:nonlinear-estimate-resonance-1} \lsm & \normb{N_{2k}A_{N_{2k}}\partial_x \wb{w}_{\lsm N_{2k}}}_{L_t^1 L_x^2 l_{N_{2k}}^2} \\
\lsm & \normb{N_{2k}\norm{w_{N}}_{l_{N}^\I}^{2k-2} \norm{Nw_{N}}_{l_{N}^2} w_{N_{2k}}\partial_x \wb{w}_{\lsm N_{2k}}}_{L_t^1 L_x^2 l_{N_{2k}}^2} \\
\lsm & \|Nw_N\|_{L_{t,x}^6l_N^2}^2\|\partial_xw_{\lesssim N}\|_{L_t^4L_x^{\infty}l_N^{\infty}}\|w_N\|_{L_{t,x}^6l_N^{\infty}}\|w_{ N}\|_{L_t^4L_x^{\infty}l_N^{\infty}}\|w_N\|_{L_{t,x}^{\infty}l_N^{\infty}}^{2k-4}\\
\lesssim& \|\partial_xw\|_{L_{t,x}^6}^2\|\partial_xw\|_{L_t^4L_x^{\infty}}\|w\|_{L_{t,x}^6}\|w\|_{L_t^4L_x^{\infty}}\|w\|_{L_{t,x}^{\infty}}^{2k-4}\\
\lsm & \|w\|_{X(I)}^{2k-2}\|w\|_{Y(I)}^{3}.
}
This gives the desired bound for \eqref{est:nonlinear-estimate-resonance-1}.

For the term \eqref{est:nonlinear-estimate-resonance-2}, by Lemmas \ref{lem:strichartz} and \ref{lem:Bernstein}, we have
\begin{align*}
\eqref{est:nonlinear-estimate-resonance-2} \lesssim& \sum_{N:N\leq N_0}\Big\|\jb{\partial_x}\int_{t_0}^{t}e^{i(t-s)\pd_x^2}P_{\sim N}(|w_{\ll N}|^{2k-2} w_{\ll N}^2\pd_x\wb w_N)ds\Big\|_{{L_t^{\infty}L_x^2}\cap L_t^{4}L_x^{\infty}\cap L_{t,x}^6}\nonumber\\
	\lesssim&  \sum_{N:N\leq N_0}\|\jb{\partial_x}P_{\sim N}(|w_{\ll N}|^{2k-2} w_{\ll N}^2\pd_x\wb w_N)\|_{{L_t^{1}L_x^2}}\nonumber\\
	\lesssim& N_0 \sum_{N:N\leq N_0} \|w_{\ll N}\|_{L_{t,x}^6}^2\|w_{\ll N}\|_{L_{t,x}^{\infty}}^{2k-4}\|w_{\ll N}\|_{L_t^{4}L_x^{\infty}}^2\|\partial_xw_N\|_{L_{t,x}^6}\nonumber\\
	\lesssim& N_0^2 \|w\|_{L_{t,x}^6}^2\|\partial_xw\|_{L_{t,x}^6}\|w\|_{L_t^{4}L_x^{\infty}}^2\|w\|_{L_t^{\infty}H_x^{1}}^{2k-4}\nonumber\\
	\lesssim& N_0^2 \|w\|_{X(I)}^{2k-2}\|w\|_{Y(I)}^3,
\end{align*}
which finishes the proof of this lemma.
\end{proof}
\subsection{Higher order terms}
\begin{lem}[Higher order terms]\label{lem:nonlinear-estimate-higher-order}
Let $I \subset\R$ be an interval containing $t_0$. Then,
\EQn{\label{eq:nonlinear-estimate-higher-order}
& \sum_{j=1}^{2k+1}\normbb{\int_{t_0}^t e^{i(t-s)\pd_x^2}\Om_j(w,(F_1+F_2+F_3)(w))(s) \dd s}_{X\cap Y(I)} \\
\lsm & \|w\|_{X(I)}^{4k-5}\|w\|_{Y(I)}^6+ \|w\|_{X(I)}^{6k-2}\|w\|_{Y(I)}^3+\|w\|_{X(I)}^{4k-2}\|w\|_{Y(I)}^{3}.
}
\end{lem}
\begin{proof}
By the Strichartz's estimates, for $1\le j\le 2k+1$,
\EQn{\label{est:nonlinear-estimate-higher-order-1}
\normbb{\int_{t_0}^t e^{i(t-s)\pd_x^2}\Om_j(w,(F_1+F_2+F_3)(w))(s) \dd s}_{X\cap Y(I)} \lsm \normb{\jb{\pd_x}\Om_j(w,(F_1+F_2+F_3)(w))}_{L_t^1 L_x^2}.
}	

$\bullet$ \textbf{Case A: $1\le j\le 2k$.} We split the (RHS) of \eqref{est:nonlinear-estimate-higher-order-1} into the following three terms:
\EQnnsub{
\text{(RHS) of }\eqref{est:nonlinear-estimate-higher-order-1}
	\lesssim&\normb{\jb{\pd_x}\Om_j(w,F_1(w))}_{L_t^1 L_x^2} \label{est:nonlinear-estimate-higher-order-1-small-j-1}\\
	&+\normb{\jb{\pd_x}\Om_j(w,F_2(w))}_{L_t^1 L_x^2} \label{est:nonlinear-estimate-higher-order-1-small-j-2}\\
	&+\normb{\jb{\pd_x}\Om_j(w,F_3(w))}_{L_t^1 L_x^2}. \label{est:nonlinear-estimate-higher-order-1-small-j-3}
}
For \eqref{est:nonlinear-estimate-higher-order-1-small-j-1}, by Lemma \ref{lem:Coifman-Meyer}, \eqref{8167}, and Sobolev's embedding, we get
\begin{align*}
\eqref{est:nonlinear-estimate-higher-order-1-small-j-1}\lesssim& \sum_{N:N\geq N_0}\Big\|\|w_N\|_{L_x^{2}}\|w_{\ll N}\|_{L_x^{\infty}}^{2k-1}\|P_{\ll N}(w\int_{-\infty}^x(\partial_y w\bar{w})^2|w|^{2k-4}dy)\|_{L_x^{\infty}}\Big\|_{L_t^1}\\
\lesssim &\sum_{N:N\geq N_0}\|w_N\|_{L_t^{\infty}L_x^{2}}\|w_{\ll N}\|_{ L_{t,x}^{\infty}}^{2k} \Big\|\big(\int_{-\infty}^x(\partial_y w\bar{w})^2|w|^{2k-4}dy\big)_{\ll N}\Big\|_{L_t^1L_x^{\infty}}\\
\lesssim &N_0^{-1}\|w\|_{ L_t^{\infty}H_x^1}^{2k+1}\|w\|_{Y(I)}^{6}\|w\|_{X(I)}^{2k-6}\\
\lesssim&N_0^{-1}\|w\|_{X(I)}^{4k-5} \|w\|_{Y(I)}^6.
\end{align*}
Similarly, for \eqref{est:nonlinear-estimate-higher-order-1-small-j-2} and \eqref{est:nonlinear-estimate-higher-order-1-small-j-3}, we have
\begin{align*}
\eqref{est:nonlinear-estimate-higher-order-1-small-j-2}\lesssim& \sum_{N:N\geq N_0}\Big\|\|w_N\|_{L_x^6}\|w_{\ll N}\|_{L_x^6}^2\|w_{\ll N}\|_{L_x^{\infty}}^{2k-3}\|P_{\ll N}(|w|^{4k}w)\|_{L_x^{\infty}}\Big\|_{L_t^1}\\
	\lesssim &\sum_{N:N\geq N_0}\|w_N\|_{ L_{t,x}^6}\|w_{\ll N}\|_{ L_{t,x}^6}^2 \|w_{\ll N}\|_{ L_t^4L_x^{\infty}}^2\|w_{\ll N}\|_{ L_{t,x}^{\infty}}^{6k-4}\\
	\lesssim & N_0^{-1}\|\partial_x w \|_{ L_{t,x}^6}\|w\|_{ L_{t,x}^6}^2\|w\|_{ L_t^4L_x^{\infty}}^2\|w\|_{ L_t^{\infty}H_x^1}^{6k-4}\\
	\lesssim&  N_0^{-1} \|w\|_{X(I)}^{6k-2} \|w\|_{Y(I)}^3,
\end{align*}
and
\begin{align*}
\eqref{est:nonlinear-estimate-higher-order-1-small-j-3}
\lesssim& \sum_{N:N\geq N_0}\Big\|\|w_N\|_{L_x^6}\|w_{\ll N}\|_{L_x^{\infty}}^{2k-1}\|P_{\ll N}(|w|^{2k-2}w^2 \partial_x \wb{w})\|_{L_x^3}\Big\|_{L_t^1}\\
\lesssim &\sum_{N:N\geq N_0}\|w_N\|_{ L_{t,x}^6}\| \partial_x w_{\ll N}\|_{ L_{t,x}^6} \|w_{\ll N}\|_{ L_{t,x}^6}\|w_{\ll N}\|_{ L_t^4L_x^{\infty}}^2\|w_{\ll N}\|_{ L_{t,x}^{\infty}}^{4k-4}\\
\lesssim & N_0^{-1}\|\partial_x w \|_{ L_{t,x}^6}^2\|w\|_{ L_{t,x}^6}\|w\|_{ L_t^4L_x^{\infty}}^2\|w\|_{ L_t^{\infty}H_x^1}^{4k-4}\\
\lesssim&  N_0^{-1} \|w\|_{X(I)}^{4k-2} \|w\|_{Y(I)}^3.
\end{align*}

Collecting the above three estimates, we have that
\EQn{\label{8168}
& \sum_{j=1}^{2k}\normbb{\int_{t_0}^t e^{i(t-s)\pd_x^2}\Om_j(w,(F_1+F_2+F_3)(w))(s) \dd s}_{X\cap Y(I)} \\
& \lesssim N_0^{-1} \brkb{ \|w\|_{X(I)}^{4k-5} \|w\|_{Y(I)}^6+ \|w\|_{X(I)}^{6k-2} \|w\|_{Y(I)}^3+\|w\|_{X(I)}^{4k-2} \|w\|_{Y(I)}^3}.
}

$\bullet$ \textbf{Case B: $j= 2k+1$.} We also consider the following three terms,
\EQnnsub{
\text{(RHS) of }\eqref{est:nonlinear-estimate-higher-order-1}
\lesssim&\normb{\jb{\pd_x}\Om_{2k+1}(w,F_1(w))}_{L_t^1 L_x^2} \label{est:nonlinear-estimate-higher-order-1-large-j-1}\\
&+\normb{\jb{\pd_x}\Om_{2k+1}(w,F_2(w))}_{L_t^1 L_x^2} \label{est:nonlinear-estimate-higher-order-1-large-j-2}\\
&+\normb{\jb{\pd_x}\Om_{2k+1}(w,F_3(w))}_{L_t^1 L_x^2}. \label{est:nonlinear-estimate-higher-order-1-large-j-3}
}
For the term \eqref{est:nonlinear-estimate-higher-order-1-large-j-1}, by Minkowski's inequality, \eqref{8167}, and Lemmas \ref{lem:Bernstein},  \ref{lem:littlewood-Paley} and \ref{lem:Coifman-Meyer}, we get
\begin{align*}
\eqref{est:nonlinear-estimate-higher-order-1-large-j-1} \lesssim& \Big\|\jb{\pd_x}\sum_{N:N\geq N_0}P_{\thicksim N}\Om_{2k+1}\big(w,F_1(w)\big)\Big\|_{L_t^1 L_x^2}\\
\lesssim& \|NP_{\thicksim N}\Om_{2k+1}\big(w_{\ll N},P_NF_1(w)\big)\|_{L_t^1 L_x^2l_N^2}\\
	\lesssim& \|w_{\ll N}\|_{L_{t,x}^{\infty}l_N^{\infty }}^{2k}\|(w\int_{-\infty}^x(\partial_y w\wb{w})^2|w|^{2k-4}dy)_{ N}\|_{L_t^1 L_x^2l_N^2}\\
	\lesssim&\|w\|_{L_t^{\infty}H_x^1}^{2k}\|w\int_{-\infty}^x(\partial_y w\wb{w})^2|w|^{2k-4}dy\|_{L_t^1 L_x^2}\\
	\lesssim&\|w\|_{X(I)}^{4k-5}\|w\|_{Y(I)}^6.
\end{align*}
Similarly, for \eqref{est:nonlinear-estimate-higher-order-1-large-j-2} and \eqref{est:nonlinear-estimate-higher-order-1-large-j-3}, we have
\begin{align*}
	\eqref{est:nonlinear-estimate-higher-order-1-large-j-2}
	\lesssim& \|NP_{\thicksim N}\Om_{2k+1}\big(w_{\ll N},P_NF_2(w)\big)\|_{L_t^1 L_x^2l_N^2}\\
	\lesssim& \|w_{\ll N}\|_{L_{t,x}^{\infty}l_N^{\infty }}^{2k}\|(|w|^{4k}w)_{ N}\|_{L_t^1L_x^2l_N^2}\\
	\lesssim&\|w\|_{L_t^{\infty}H_x^1}^{2k}\||w|^{4k}w\|_{L_t^1 L_x^2}\\
	\lesssim&\|w\|_{L_t^{\infty}H_x^1}^{2k}\|w\|_{L_{t,x}^6}^3\|w\|_{L_t^4L_x^{\infty}}^2\|w\|_{L_{t,x}^{\infty}}^{4k-4}\\
	\lesssim&\|w\|_{X(I)}^{6k-2}\|w\|_{Y(I)}^3,
\end{align*}
and
\begin{align*}
\eqref{est:nonlinear-estimate-higher-order-1-large-j-3}
\lesssim& \|NP_{\thicksim N}\Om_{2k+1}\big(w_{\ll N},P_NF_3(w)\big)\|_{L_t^1 L_x^2l_N^2}\\
	\lesssim& \|w_{\ll N}\|_{L_{t,x}^{\infty}l_N^{\infty }}^{2k}\|(|w|^{2k-2}w^2 \partial_x \wb{w})_{ N}\|_{L_t^1L_x^2l_N^2 }\\
	\lesssim&\|w\|_{L_t^{\infty}H_x^1}^{2k}\||w|^{2k-2}w^2 \partial_x \wb{w}\|_{L_t^1 L_x^2}\\
	\lesssim&\|w\|_{L_t^{\infty}H_x^1}^{2k}\|w\|_{L_{t,x}^6}^{2}\|w\|_{L_t^4L_x^{\infty}}^2\|\partial_x w\|_{L_{t,x}^6}\|w\|_{L_{t,x}^{\infty}}^{2k-4}\\
	\lesssim&\|w\|_{X(I)}^{4k-2}\|w\|_{Y(I)}^{3}.
\end{align*}

By the above three estimates, we obtain
\EQn{\label{8179}
& \normbb{\int_{t_0}^t e^{i(t-s)\pd_x^2}\Om_{2k+1}(w,(F_1+F_2+F_3)(w))(s) \dd s}_{X\cap Y(I)} \\
& \lesssim  \|w\|_{X(I)}^{4k-5}\|w\|_{Y(I)}^6+ \|w\|_{X(I)}^{6k-2}\|w\|_{Y(I)}^3+\|w\|_{X(I)}^{4k-2}\|w\|_{Y(I)}^{3}.
}
Then, \eqref{eq:nonlinear-estimate-higher-order} follows by \eqref{8168} and \eqref{8179}.
\end{proof}

\section{Existence of wave operator}
Our objective of this section is to prove Theorem \ref{main theorem1}. We first reduce the final data problem of \eqref{GDNLS} to the equation  \eqref{eq:gdnls-gauge} after gauge transform. The existence of wave operator for \eqref{eq:gdnls-gauge} is given as follows, where we only consider the forward-in-time case.
\begin{prop}\label{main prop}
Let $\sigma=k\in \mathbb{Z}$ and $k\geq 3$. Let $V_+\in H_x^1$. Then, there exists $T\gg 1$ depending on the profile of $V_+$, such that there exists a unique solution $w\in C([T, +\infty);H_x^1(\R))$ to \eqref{eq:gdnls-gauge}, satisfying
\begin{align*}
\lim_{t\ra+\I}\|w(t,x)-e^{it\pd_x^2 }V_+\|_{H_x^1}=0.
\end{align*}
\end{prop}

Next, we provide the proof of Theorem \ref{main theorem1} using Proposition \ref{main prop}.

\begin{proof}[Proof of Theorem \ref{main theorem1}]
We only prove the forward-in-time case, and it suffices to prove
 \begin{align}\label{8223}
  \|u(t)-e^{it\pd_x^2}V_+\|_{H_x^1}\rightarrow 0, \quad \mbox{as }t\rightarrow +\infty.
  \end{align}
By Proposition \ref{main prop}, we have the existence and uniqueness of $w$ and
\begin{align}\label{8222}
	\|w(t)-e^{it\pd_x^2}V_+\|_{H_x^1}\rightarrow 0, \quad \mbox{as }t\rightarrow +\infty.
\end{align}
By the density argument, we are reduced to prove \eqref{8223} assuming $w$ satisfies
 \begin{align}\label{8241}
 \|w\|_{L_x^{2k}\cap L_x^{\infty}}\rightarrow 0, \quad \mbox{as }t\rightarrow +\infty.
  \end{align}
  Moreover, recall the gauge transform in \eqref{8242}, we have
  \begin{align}\label{8224}
  	\|u(t)-e^{it\pd_x^2}V_+\|_{H_x^1}=&\|w(t,x)-e^{\frac i2\int_{-\infty}^x|w(t,y)|^{2k} dy}e^{it\pd_x^2}V_+\|_{H_x^1}\nonumber\\
  	\lesssim& \|w(t,x)-e^{it\pd_x^2}V_+\|_{H_x^1}+\|(e^{\frac i2\int_{-\infty}^x|w(t,y)|^{2k} dy}-1)e^{it\pd_x^2}V_+\|_{H_x^1}\nonumber\\
  	\lesssim& \|w(t,x)-e^{it\pd_x^2}V_+\|_{H_x^1}+\|e^{\frac i2\int_{-\infty}^x|w(t,y)|^{2k} dy}-1\|_{L_x^{\infty}}\|e^{it\pd_x^2}V_+\|_{H_x^1}\nonumber\\
  	&+\Big\|\partial_x\big(e^{\frac i2\int_{-\infty}^x|w(t,y)|^{2k} dy}-1\big)\Big\|_{L_x^{\infty}}\|e^{it\pd_x^2}V_+\|_{L_x^2}.
  \end{align}
By \eqref{8241}, we can easily obtain that when $t\rightarrow +\infty$,
 \begin{align}\label{f-1}
 \Big\|\int_{-\infty}^x|w(t,y)|^{2k} dy\Big\|_{L_x^{\infty}}\lesssim \|w\|_{L_x^{2k}}^{2k}\rightarrow 0,
  \end{align}
  and
  \begin{align}\label{f-2}
 \Big\|\partial_x\big(e^{\frac i2\int_{-\infty}^x|w(t,y)|^{2k} dy}-1\big)\Big\|_{L_x^{\infty}}\lesssim \|w\|_{L_x^{\infty}}^{2k}\rightarrow 0.
  \end{align}
Hence, \eqref{8223} follows from the estimates \eqref{8222}, \eqref{8224}, \eqref{f-1} and \eqref{f-2}.
\end{proof}

Now, it suffices to prove Proposition \ref{main prop}.
\begin{proof}[Proof of Proposition \ref{main prop}]
By the Strichartz's estimates, we have that
\begin{align*}
	\|e^{it\pd_x^2}V_+\|_{X(I)}\le C\|V_+\|_{H_x^1}:=R,
\end{align*}
and for any $0<\delta\ll 1$, there exists $T\gg 1$, such that
\begin{align*}
	\|e^{it\pd_x^2}V_+\|_{Y([T,+\I))}\le\delta.
\end{align*}
We assume that $R\ge1$ and $\de\le R$ without loss of generality.
Let $t_n>T$ such that $\lim_{n\ra\I} t_n = +\I$. Assume $w_n$ is the solution to \eqref{eq:gdnls-normal-form} with $t_0=t_n$ and $w_0=e^{it_n\pd_x^2}V_+$, namely
\EQn{\label{eq:gdnls-normal-form-tn}
w_n = & e^{it\pd_x^2} V_+ - i\int_{t_n}^t e^{i(t-s)\pd_x^2}(F_1+F_2)(w_n(s)) \dd s  \\
& - ke^{i(t-t_n)\pd_x^2} \Om(e^{it_n\pd_x^2}V_+) +k \Om(w_n)+ k\int_{t_n}^t e^{i(t-s)\pd_x^2} R(w_n(s)) \dd s \\
& -ik\sum_{j=1}^{2k+1}\int_{t_n}^t e^{i(t-s)\pd_x^2}\Om_j(w_n,(F_1+F_2+F_3)(w_n))(s) \dd s.
}
Denote the right hand side of the above equation as $\Phi_n(w_n)$. Let $I=[T,+\I)$.
Take the work space as
\EQ{
B_{R,\de}:=\fbrk{w_n\in C(I;H_x^1(\R)): \norm{w_n}_{X(I)}\le 2R\text{, and }\norm{w_n}_{Y(I)}\le 2\de}.
}
By Lemma \ref{lem:nonlinear-estimate-f1f2},
\EQn{\label{eq:contraction-wn-f1f2}
& \normbb{\int_{t_n}^t e^{i(t-s)\pd_x^2}(F_1+F_2)(w_n(s)) \dd s}_{X\cap Y(I)} \\ \lsm & \|w_n\|_{X(I)}^{2k-5}\|w_n\|_{Y(I)}^6+\|w_n\|_{X(I)}^{2k-2}\|w_n\|_{Y(I)}^3 + \|w_n\|_{X(I)}^{4k-2}\|w_n\|_{Y(I)}^3 \\
\lsm & R^{2k-5}\de^6+R^{2k-2}\de^3 + R^{4k-2}\de^3 \\
\lsm&_{R}  \de^2.
}
By the Strichartz's estimates and Lemma \ref{lem:nonlinear-estimate-boundary},
\EQ{
\normb{e^{i(t-t_n)\pd_x^2} \Om(e^{it_n\pd_x^2}V_+)}_{X\cap Y(I)} \lsm &  N_0^{-1} \normb{ e^{it_n\pd_x^2}V_+}_{L_x^\I}^{2k} \normb{ e^{it_n\pd_x^2}V_+}_{H_x^1} \\
\lsm & N_0^{-1} R \normb{ e^{it_n\pd_x^2}V_+}_{L_x^\I}^{2k}.
}
Furthermore, by density argument, we can obtain that
\EQn{\label{eq:contraction-wn-boundary-linear-limit}
\lim_{n\ra\I} \normb{e^{i(t-t_n)\pd_x^2} \Om(e^{it_n\pd_x^2}V_+)}_{X\cap Y(I)} = 0.
}
Hence, for sufficiently large $n$,
\EQn{\label{eq:contraction-wn-boundary-linear}
\normb{e^{i(t-t_n)\pd_x^2} \Om(e^{it_n\pd_x^2}V_+)}_{X\cap Y(I)} \le \frac14\de.
}
Moreover, by Lemma \ref{lem:nonlinear-estimate-boundary},
\EQn{\label{eq:contraction-wn-boundary-x}
\norm{\Om(w_n)}_{X(I)} \lsm N_0^{-\frac12-} \norm{w_n}_{X(I)}^{2k+1} \lsm N_0^{-\frac12-} R^{2k+1},
}
and
\EQn{\label{eq:contraction-wn-boundary-y}
\norm{\Om(w_n)}_{Y(I)} \lsm N_0^{-1} \norm{w_n}_{X(I)}^{2k} \norm{w_n}_{Y(I)} \lsm N_0^{-1} R^{2k} \de .
}
By Lemma \ref{lem:nonlinear-estimate-resonance},
\EQn{\label{eq:contraction-wn-resonance}
\normbb{\int_{t_n}^t e^{i(t-s)\pd_x^2} R(w_n(s)) \dd s}_{X\cap Y(I)} \lsm N_0^2 \norm{w_n}_{X(I)}^{2k-2}\norm{w_n}_{Y(I)}^{3} \lsm N_0^2 R^{2k-2}\de^{3} \lsm_{N_0,R} \de^2.
}
By Lemma \ref{lem:nonlinear-estimate-higher-order},
\EQn{\label{eq:contraction-wn-higher-order}
& \sum_{j=1}^{2k+1} \normb{\int_{t_n}^t e^{i(t-s)\pd_x^2}\Om_j(w_n,(F_1+F_2+F_3)(w_n))(s) \dd s}_{X\cap Y(I)} \\
\lsm & \|w_n\|_{X(I)}^{4k-5}\|w_n\|_{Y(I)}^6+ \|w_n\|_{X(I)}^{6k-2}\|w_n\|_{Y(I)}^3+\|w_n\|_{X(I)}^{4k-2}\|w_n\|_{Y(I)}^{3} \\
\lsm & R^{4k-5}\de^6+ R^{6k-2}\de^3+R^{4k-2}\de^{3} \\
\lsm & _R \de^2.
}
By \eqref{eq:contraction-wn-f1f2}, \eqref{eq:contraction-wn-boundary-linear}, \eqref{eq:contraction-wn-boundary-x}, \eqref{eq:contraction-wn-resonance}, \eqref{eq:contraction-wn-higher-order}, for $w_n \in B_{R,\de}$ and sufficiently large $n$,
\EQ{
\norm{\Phi_n(w_n)}_{X(I)} \le & R + C(R)\de^2 + \frac14\de + C N_0^{-\frac12-} R^{2k+1} + C(N_0,R)\de^2 + C(R) \de^2 \\
\le & R + \brkb{\frac14+C(N_0,R)\de}\de + C N_0^{-\frac12-} R^{2k+1}.
}
First, we take $N_0=N_0(R)$ such that
\EQ{
C N_0^{-\frac12-} R^{2k} \le\frac12.
}
Then, we take $\de=\de(N_0, R)=\de(R)$ such that
\EQ{
C(N_0, R)\de \le \frac14.
}
Therefore,
\EQn{\label{eq:contraction-wn-x}
\norm{\Phi_n(w_n)}_{X(I)} \le & R + \frac12\de + \frac12 R
\le  2R.
}
Similarly, by \eqref{eq:contraction-wn-f1f2}, \eqref{eq:contraction-wn-boundary-linear}, \eqref{eq:contraction-wn-boundary-y}, \eqref{eq:contraction-wn-resonance}, \eqref{eq:contraction-wn-higher-order}, for $w_n \in B_{R,\de}$ and sufficiently large $n$,
\EQn{\label{eq:contraction-wn-y}
\norm{\Phi_n(w_n)}_{Y(I)} \le & \de + C(R)\de^2 + \frac14\de + C N_0^{-1} R^{2k} \de + C(N_0,R)\de^2 + C(R) \de^2 \\
\le & \de + \brkb{\frac14+C(N_0,R)\de + C N_0^{-1}R^{2k}}\de  \\
\le & 2\de.
}
Therefore, combining \eqref{eq:contraction-wn-x} and \eqref{eq:contraction-wn-y}, we have that $\Phi_n:B_{R,\de}\ra B_{R,\de}$. The contraction mapping then follows using similar argument. Therefore, we obtain the existence and uniqueness of $w_n\in C([T,+\I);H_x^1(\R))$ to the equation \eqref{eq:gdnls-normal-form-tn}.

Next, we take the limit in $n$. In the view of \eqref{eq:contraction-wn-boundary-linear-limit}, we expect the limit solution solves the following equation:
\EQn{\label{eq:gdnls-normal-form-t-infinity}
w = & e^{it\pd_x^2} V_+ - i\int_{+\I}^t e^{i(t-s)\pd_x^2}(F_1+F_2)(w(s)) \dd s  \\
& + k\Om(w)+ k\int_{+\I}^t e^{i(t-s)\pd_x^2} R(w(s)) \dd s \\
& -ik\sum_{j=1}^{2k+1}\int_{+\I}^t e^{i(t-s)\pd_x^2}\Om_j(w,(F_1+F_2+F_3)(w))(s) \dd s.
}
In fact, using the similar method when we solves \eqref{eq:gdnls-normal-form-tn}, we are able to prove that the equation \eqref{eq:gdnls-normal-form-t-infinity} is well-posed in
\EQ{
B'_{R,\de}:=\fbrk{w\in C([T,\I);H_x^1(\R)): \norm{w}_{X(I)}\le 2R\text{, and }\norm{w}_{Y(I)}\le 2\de},
}
for the above $T$. Moreover, take a difference between the equations \eqref{eq:gdnls-normal-form-tn} and \eqref{eq:gdnls-normal-form-t-infinity}, and use the nonlinear estimates as above, then we get that
\EQ{
\lim_{n\ra +\I} \norm{w_n-w}_{X\cap Y([T,+\I))}=0.
}
Furthermore, since $w_n$ solves the equation \eqref{eq:gdnls-normal-form-tn}, it is also the solution to \eqref{eq:gdnls-gauge} with Cauchy data $w(t_n)=e^{it_n\pd_x^2} V_+$. Therefore, the limit solution $w$ to \eqref{eq:gdnls-normal-form-t-infinity} is also the solution to \eqref{eq:gdnls-gauge} satisfying
\EQ{
\lim_{t\ra +\I}\normb{w(t)-e^{it\pd_x^2}V_+}_{H_x^1(\R)} =0.
}
Then, we finish the proof of Proposition \ref{main prop}.
\end{proof}

\section*{Acknowledgements}
The authors would like to express their gratitude to Professor Yifei Wu for his valuable comments and insightful suggestions.

\vskip 0.2cm

\end{document}